\documentclass[12pt,draftcls,onecolumn]{IEEEtran}
\IEEEoverridecommandlockouts                              
\usepackage{isolatin1}

\usepackage{times} 
\usepackage{pdfsync} 
\usepackage{epsfig}
\usepackage{epstopdf}
\usepackage{amsmath} 
\usepackage{amssymb}
\usepackage{color} 
\usepackage{cite} 
\usepackage{array}
\usepackage{mathrsfs}
\usepackage{fix2col}

\newcommand{\eq}{\triangleq}

\newcommand{\Prob}{\mathbf{Pr}}
\newcommand{\E}{\mathbf{E}}

\newcommand{\field}[1]{\mathbb{#1}} 
\newcommand{\R}{\field{R}}
\newcommand{\N}{\field{N}}

\newcommand{\K}{\mathcal{K}}

\newcommand{\hfs}{\hfill\ensuremath{\square}}

\newtheorem{theorem}{Theorem}[section]
\newtheorem{coro}[theorem]{Corollary}

\newtheorem{lem}[theorem]{Lemma}

\newtheorem{ex}{Example}[section]

\newtheorem{ass}{Assumption}
\newtheorem{defn}{Definition}

\title{Anytime Control using Input Sequences with Markovian Processor Availability}

\author{Daniel E.\ Quevedo,
\thanks{D. Quevedo is with the School of Electrical Engineering \& Computer Science, University of Newcastle. Email: {\tt dquevedo@ieee.org}. Research supported by Australian Research
  Council's Discovery Projects scheme (project DP0988601).
}
Wann-Jiun Ma,
\thanks{W. J. Ma and V. Gupta are with the Department of Electrical Engineering, University of Notre Dame. Email: {\tt wma1@nd.edu, vgupta2@nd.edu}. Research supported in part by NSF awards 0846631 and 0834771.
}
and Vijay Gupta
\thanks{A preliminary version of parts of this note was
  presented  at the 1st Australian Control Conference, Melbourne; see \cite{quegup11a}.
}
}

\begin{document}
\maketitle

\begin{abstract}
We study an anytime control algorithm for situations where the  processing
resources available for control are time-varying in an a priori unknown
fashion. Thus, at times,  processing resources are insufficient to calculate
control inputs. To address this issue, the algorithm calculates sequences of
tentative future control inputs whenever possible, which are then buffered for
possible future use.  We assume that the processor availability is correlated so
that  the number of control inputs calculated at any time step is described by a
Markov chain. Using a Lyapunov function based approach we derive
sufficient conditions for stochastic stability of the closed loop.  
\end{abstract}

\section{Introduction}
Recently, many works have appeared that consider the impact of limited or
time-varying processing power on control algorithms. Such problems arise
naturally in cyberphysical and embedded systems where the control algorithm may
be just one of many tasks being executed by the processor. Thus, McGovern and
Feron~\cite{MF98,MF99} considered the question of bounding the processing time
that is required to solve the optimization problem in model predictive control
to a specified accuracy. Henriksson \emph{et al}~\cite{ha04,hcaa02} studied the
trade-off inherent in solving the optimization problem exactly (thus, obtaining
the control input sequence more precisely) and in solving the problem more
often. Event-triggered and self-triggered control, and online sampling,
e.g.,~\cite{t07,wl07,vmb09,cvmc10} have also been proposed as a means to ensure
less demand on the processor on average by calculating the control input on
demand in a non-periodic fashion.  

In this note, we are interested in anytime control algorithms. Such algorithms calculate a coarse control input even with limited processing resources. As more processing resources become available, the input is refined. The process can be terminated at any time by the processor. The quality of control input is thus time-varying, but no control input is obtained only rarely. Various anytime algorithms for linear processors and controllers have been proposed in the literature~\cite{bb04,GFB07,g10}. For non-linear plants, we recently proposed anytime algorithms  based on computing sequences of potential (tentative) future control values~\cite{quegup12a}. At the instances when more processing power is available, a longer sequence is calculated. This provides a buffer against the time steps when the processor power is not enough to calculate an input. Since the control values in the sequence are calculated by reutilising already computed values, the algorithm does not assume a priori knowledge of processor availability.

However, with the exception of \cite{GFB07} and \cite{quegup12a}, the analysis
in these works largely considered the processor availability to be described by
an independent and identically distributed sequence. In particular,~\cite{quegup12a}
had a brief discussion when the processor availability sequence is described by
a (hidden)
Markov chain;   the memory arose through the concept of `processor
states' which are not directly related to how many control values can be
calculated.  In the current work, we replace this model by a more direct one,
where the processor availability for the control 
task, and hence the number of tentative control values that can be calculated at
each time step,  forms a Markov Chain. More importantly, we provide a new
analysis technique, that at least for a class of models, is less conservative than the technique
in~\cite{quegup12a}. Intuitively, the proposed technique considers the `average' case of
processor availability to analyze a random-time drift condition, as compared
to the `worst case' analysis in~\cite{quegup12a}.  Sufficient conditions for
stochastic stability with and without the anytime control algorithm are provided
and compared with the conditions in \cite{quegup12a}. We also analyze the
robustness of these conditions with respect to presence of process noise. A
preliminary version of parts of the present manuscript can be found
in~\cite{quegup11a}. 

\par The paper is organized as follows: In
Section~\ref{sec:prob_form}, we present the control design problem
studied. In Section~\ref{sec:sequ-based-anyt}, we revise the  anytime algorithm
of\cite{quegup12a} to be studied. Section~\ref{sec:analys-via-assoc} presents a novel model
for analyzing the resulting closed loop  when the processor availability is
Markovian. Section~\ref{sec:analysis} presents the stability analysis with this
model.  Section~\ref{sec:relat-prev-results} compares our results with those
in\cite{quegup12a}. Section~\ref{sec:notes-rubustness} provides  robust stability analysis in the presence of process noise. Numerical simulations are documented in
Section~\ref{sec:case-studies}.  
 Section~\ref{sec:conclusions} draws conclusions.

\paragraph*{Notation}
We write $\N$ for $\{1, 2, 3, \ldots\}$, $\N_0$ for $\N \cup \{0\}$ and $\N_{n}^{m}= \{n,n+1, \ldots,m\},$ for given integers $n\leq m$. $\R$ are the real numbers and
$\R_{\geq 0}$ the nonnegative real numbers. The $p\times p$ identity matrix is
denoted by $I_p$ and the $p\times q$ matrix of all ones is denoted by $I_{p\times q}$, whereas $0_p = 0 I_p$ and $\mathbf{0}_p$ is the  all-zeroes (column)
vector in $\R^p$. The notation
$\{x\}_{\K}$ stands for 
$\{x(k) \;\colon k \in \K\}$.
We  adopt the convention $\sum_{k=\ell_1}^{\ell_2}a_k = 0$  if $\ell_1 > \ell_2$
and irrespective of $a_k\in\R$. The superscript $^T$  refers to transpose. The
Euclidean norm of a vector $x$ is denoted by $|x|=\sqrt{x^Tx}$.
 A
function $\varphi\colon \R_{\geq 0}\to \R_{\geq 0}$ is of
\emph{class-}$\mathscr{K}_\infty$ ($\varphi \in \mathscr{K}_\infty$), if it is
continuous, zero at zero, strictly increasing, and  unbounded.
  The probability of event
 $\Omega$ is 
 $\Prob\{\Omega \}$ and the conditional probability of $\Omega$ given
  $\Gamma$ is $\Prob\{\Omega\,|\,\Gamma \}$. The  expected value of  $\nu$ given 
  $\Gamma$, is denoted by  $\E\{\nu  \,|\, \Gamma \}$; for the
  unconditional 
  expectation  we  write $\E\{\nu\}$. An $m\times n$ matrix $M$ whose $(i,j)$-th element is $m_{ij}$ is denoted by $M=\left[m_{ij}\right]_{m\times n}$.

\section{Control with Random  Processor Availability}
\label{sec:prob_form}
Consider a discrete-time non-linear plant that evolves as 
\begin{equation}
  \label{eq:15}
  x(k+1) = f(x(k),u(k)),\quad k\in\N_0,
\end{equation}
where the state $x(.)\in \R^n$ and the control input $u(.)\in \R^p.$ We assume that the origin is an equilibrium point of the plant, so that  $f(\mathbf{0}_n,\mathbf{0}_p)=\mathbf{0}_n$. The initial state $x(0)$ is arbitrary. 
Given the stochastic processor availability model that we assume (as described below), the plant can evolve in open loop for arbitrarily long times. For general non-linear plants, the state may thus assume a value such that no possible control sequence can stabilize the process. To prevent this eventuality, we  assume that~\eqref{eq:15} is globally controllable via state feedback. 
\begin{ass}
\label{ass:CLF}
There exist functions  $V\colon
\R^n\to\R_{\geq 0}$,  $\varphi_1, \varphi_2\in\mathscr{K}_\infty$,  a constant
$\rho \in [0,1)$, and a control policy $\kappa
\colon \R^n\to \R^p$,
 such that  for all $x\in\R^n,$
\begin{equation}
  \label{eq:3}
  \begin{split}
    \varphi_1(|x|)\leq V(x)&\leq \varphi_2(|x|),\\
    V(f(x,\kappa(x))) &\leq \rho V(x).
  \end{split}
\end{equation}
\end{ass} 
If the plant~(\ref{eq:15}) is considered to be obtained by sampling a
continuous-time plant, it is generally assumed that the control calculation can
be completed within a fixed (and small) time-delay, say $\delta\in
(0,T_s)$.\footnote{Recall that fixed delays can be easily incorporated into the
  model~(\ref{eq:15}) by aggregating  the previous plant input to the plant
  state, see also\cite{nilber98}. For ease  of exposition, we will  use the
  standard discrete-time notation as in~(\ref{eq:15}).} However, in networked
and embedded systems, the processing resources  (e.g., processor execution
times)  for control  may vary, and, at times, be insufficient to generate a
control input within the prescribed timeout $\delta$. This can lead to instances
where the plant evolves uncontrolled, even though there was an excess of
processing resource availability (beyond what is required to calculate a single
control input) at other time instants. The anytime control algorithm we propose
makes better use of this excess availability to safeguard against the time steps
at which the processing resource was not available at all. 

Before describing the anytime algorithm, we discuss a baseline algorithm that arises from a direct implementation of the control policy $\kappa$ used in Assumption~\ref{ass:CLF}. In this algorithm, the plant input which is applied during the  interval
 $[kT_s +\delta, (k+1)T_s +\delta)$ is given by
\begin{equation}
  \label{eq:4}
  u(k)=
  \begin{cases}
    \kappa(x(k)) &\text{if sufficient computational resources to evaluate $\kappa(x(k))$ are available}\\
    &\text{between times $kT_s$ and $kT_s+\delta$,}\\
    \mathbf{0}_p &\text{otherwise.}
  \end{cases}
\end{equation}
We shall assume that the controller requires processor time to
  carry out mathematical computations. However,
  simple operations at a bit level, such as writing data into buffers, shifting
  buffer contents and setting values to zero do not require processor
  time. Similarly, input-output operations, i.e., A/D and D/A conversion are
  triggered by external asynchronous loops with a real-time clock and do not
  require that the processor be available for control. As in regular
  discrete-time control, these external loops ensure that state measurements are
  available at the instants $\{kT_s\}_{k\in\N_0}$ and that the controller
  outputs (if available) are passed on to the plant actuators at times
  $\{kT_s+\delta\}_{k\in\N_0}$, where $\delta$ is fixed.

\section{Sequence-based Anytime Control Algorithm}
\label{sec:sequ-based-anyt} 
We use the same anytime control algorithm as proposed in~\cite{quegup12a} that calculates and buffers a  sequence of tentative
future plant  inputs at time intervals when the  
controller is provided with more processing resources than are needed to evaluate
the current control input. Denote the buffer
states via $\{b\}_{\N_0}$, where  $$b(k)=\left[\begin{array}{lcr}b_{1}^{T}(k)&\cdots&b_{\Lambda}^{T}(k)\end{array}\right]^{T}\in\R^{\Lambda  p},\quad k\in\N_0,$$
for a given value $\Lambda\in\{2,3,\dots\}$ and where each $b_j(k)\in\R^p$, $j\in\N_1^\Lambda$.
Also  
define a shift matrix  
   \begin{equation*}
   S\eq
\begin{bmatrix}
    0_p & I_p& 0_p &\hdotsfor{1} &0_p\\
    \vdots & \ddots & \ddots &\ddots  & \vdots\\
     0_p &  \dots      &  0_p &I_p  & 0_p\\
    0_p & \hdotsfor{2}       &  0_p & I_p\\
 0_p &\hdotsfor{3} &  0_p
\end{bmatrix}\in\R^{\Lambda  p\times \Lambda p}.
\end{equation*}
Fig.\ 1 presents the algorithm, which we denote by A$_1$.

\linespread{1.2}
\begin{figure}[h!]
\noindent\rule{\linewidth}{0.2mm}
\begin{description}
 \itemsep2pt
\item[Step 1]:  At time $t=0$,
\par  \hspace{1cm} \textsc{set}  $b(-1)\leftarrow \mathbf{0}_{\Lambda  p}$,
  $k\leftarrow 0$; 
  
\item[Step 2]: \label{step:timek}
  \textsc{if}  $t \geq k T_s$,  \textsc{then} 
\par  \hspace{1cm}  \textsc{input}
  $x(k)$;
\par    \hspace{1cm}  \textsc{set}  $\chi\leftarrow x(k)$,  $j\leftarrow 1$,
  $b(k)\leftarrow Sb(k-1)$;
\par  \hspace{1mm}  \textsc{end}

\item[Step 3]: \label{step:repeat}
  \textsc{while} {``sufficient processor time is available'' 
  and   time $t < (k+1) T_s$ and  $j\leq \Lambda$,}
  \par\hspace{1cm} {\textsc{evaluate} $u_j(k)=\kappa(\chi)$};
\par\hspace{1cm} \textsc{if} $j=1$, \textsc{then}
\par  \hspace{3cm} 
   \textsc{output} $u_1(k)$;
\par  \hspace{3cm}   \textsc{set}
 $b(k)\leftarrow\mathbf{0}_{\Lambda  p}$;
\par\hspace{1cm} \textsc{end}
     
  \par \hspace{1cm} \textsc{set} $b_j(k)\leftarrow u_j(k)$; 
  
  \par \hspace{1cm} \textsc{if} ``sufficient processor time is not available'' 
 or  $t \geq (k+1) T_s$, \textsc{then}

 \par\hspace{3cm}  \textsc{goto} Step 5; 

\par \hspace{1.15cm}\textsc{end}    

  \par \hspace{1cm} \textsc{set} $\chi \leftarrow  f(\chi,u_j(k))$,
  $j\leftarrow j+1$;
\par \hspace{1mm} \textsc{end}

\item[Step 4]: 
  \textsc{if} $j=1$, \textsc{then}
\par  \hspace{1cm}  \textsc{output} $b_1(k)$; 
\par \hspace{3mm}\textsc{end}

\item[Step 5]: 
  \textsc{set} $k \leftarrow k+1$ and \textsc{goto} Step 2;
\end{description}

\noindent\rule{\linewidth}{0.2mm}
\caption{Anytime algorithm A$_1$, adapted from \cite{quegup12a}.}
\label{alg:1}
\end{figure}
\linespread{1.5}
 
\par Note that the algorithm essentially amounts to a dynamic state 
feedback policy with internal state variable $b(k)$. Denote by
$N(k)\in\N_0^\Lambda$ the total number of iterations of the while-loop in Step~3 which are carried out during the
interval $t\in (kT_s,(k+1)T_s)$. 
This yields: 
\begin{equation}
  \label{eq:1}
  b(k)=
  \begin{cases}
    Sb(k-1) &\text{if $N(k)=0$,}\\
    \begin{bmatrix}
      \vec{u}(k)^T&
      (\mathbf{0}_{(\Lambda -N(k)) p})^T
    \end{bmatrix}^T
&\text{if $N(k)\geq 1$,}
  \end{cases}
\end{equation}
where 
\begin{equation*}
  \vec{u}(k) = 
  \begin{bmatrix}
    u_{1}(k)\\u_{2}(k)\\ \vdots\\ u_{N(k)}(k)
  \end{bmatrix}\in\R^{N(k)\cdot p}.
\end{equation*}
 The outcomes of the process $\{N\}_{\N_0}$ affect the resultant closed loop
performance  since they determine   how many 
values which stem from the tentative control sequences $\{\vec{u}(k-\ell)\}$,
$\ell \in\N_0$ are contained in the
buffer state $b(k)$.  We refer to this quantity  as the \emph{effective buffer
  length} (at time $k\in\N_0$),  denote it as 
$\lambda (k) \in\N_0^{\Lambda}$ 
 and note that with   initial state $\lambda(-1)=0,$  
\begin{equation}
  \label{eq:19}
  \lambda(k)=
  \begin{cases}
    N(k) &\text{if $N(k)\geq 1$},\\
    \max (\lambda(k-1)-1,0)  &\text{if $N(k)=0$}.
  \end{cases}
\end{equation}

\begin{ex}
\label{ex:one}
Suppose that $\Lambda = 4$ and that the processor availability is such that
$N(0) =4$, $N(1) =0$, $N(2)=1$, $N(3)=2$.
When using the anytime algorithm A$_1$, the buffer state at times $k\in\{0,1,2,3\}$ becomes:
\begin{equation*}
    \{b(0),b(1),b(2),b(3)\}=\left\{
   \begin{bmatrix}
    {u_0(0)}\\u_1(0)\\ u_2(0)\\ u_{3}(0)
 \end{bmatrix}\!,
\begin{bmatrix}
    {u_1(0)}\\ u_2(0)\\ u_{3}(0)\\ \mathbf{0}_{p}
 \end{bmatrix}\!,
\begin{bmatrix}
    {u_0(2)}\\ \mathbf{0}_{p}\\ \mathbf{0}_{p}\\ \mathbf{0}_{p} 
 \end{bmatrix}\!,
\begin{bmatrix}
    {u_0(3)} \\ u_1(3) \\  \mathbf{0}_{p}\\ \mathbf{0}_{p}
 \end{bmatrix}
\right\} 
\end{equation*}
which gives $\lambda(0)=4$, $\lambda(1)=3$, $\lambda(2)=1$, $\lambda(3)=2$, and the plant inputs
$u(0)=u_0(0)$, $u(1)= u_1(0)$, $u(2)=u_0(2)$, and $u(3)=u_0(3)$. On the other hand,  if the  baseline-algorithm
in~\eqref{eq:4} is used, then
$  u(0) = \kappa(x(0))$, $u(1)=\mathbf{0}_{p}$, $u(2)=\kappa(x(2))$ and
$u(3)=\kappa(x(3))\}$, i.e., at time $k=1$ 
the plant input is set to zero.
This  suggests that   Algorithm A$_1$ will 
outperform the baseline algorithm. 
\hfs
\end{ex}


\section{Markov Chain Model and Analysis}
\label{sec:analys-via-assoc}
In\cite{quegup12a} we studied Algorithm A$_1$   under the assumption that
$\{N\}_{\N_0}$ is  governed by an underlying correlated processor state
process. In this work, we examine an alternative model wherein $\{N\}_{\N_0}$ is
directly described by a finite Markov Chain~\cite{kemsne60}. As we shall see in
Section~\ref{sec:relat-prev-results}, the current model enables us to develop
sufficient conditions for stability, which are less conservative than those
in\cite{quegup12a}. 

\begin{ass}
\label{ass:iid}
  The process $\{N\}_{\N_0}$  is
  a homogeneous
 Markov Chain with initial state
$N(0)=0$ and an irreducible and aperiodic transition probability matrix  $ \mathcal{Q} = [q_{ij}]_{\N_0^\Lambda\times \N_0^\Lambda}$ where
 \begin{equation}
 \label{eq:7}
 q_{ij}=\Prob\{N(k+1)=j\,|\,N(k)=i\},\quad i,j\in \N_0^\Lambda.
 \end{equation}
\end{ass}
The above model  allows for correlations in processor availability.
Fig.~\ref{fig:transitions2} depicts the transition graph for $\{N\}_{\N_0}$ 
resulting from~\eqref{eq:7} for the case where $\Lambda =2$.

\begin{figure}[t]
\centering
 \includegraphics[width=.5\textwidth]{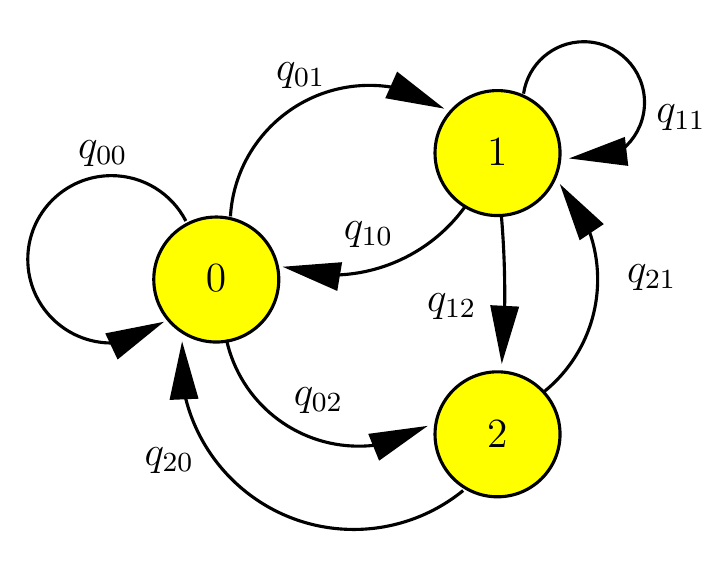}
	\caption{Transition graph of $N$ for $\Lambda =2$.}
	\label{fig:transitions2}
      \end{figure}

\subsection{Defining an aggregated process} 
We will  analyze the anytime control system through the aggregated
process $\{Z\}_{\N_0}$, where  each $$ Z(k)\eq (N(k),\lambda(k)), \quad k\in\N_0$$ 
belongs to the  set $\mathbb{S}\eq\{s_0,s_1,\dots, s_{2\Lambda-1}\}$, having elements   $s_i = (i,i),\forall i\in\N_{0}^\Lambda$ and $s_{\Lambda +j}= (0,j), \forall j \in\N_1^{\Lambda-1}.$
Clearly the  outcomes of $\{N\}_{\N_0}$ determine the
trajectory of $\{Z\}_{\N_0}$ and thereby determine whether the buffer contains
calculated control values or not.
An important  property is that, if Assumption~\ref{ass:iid} holds,
then $\{Z\}_{\N_0}$ is a  Markov Chain. The transition probabilities
 $$ p_{ij}=\Prob\{Z(k+1)=s_j\,|\,Z(k)=s_i\}, \quad (s_i,s_j)\in \mathbb{S}\times \mathbb{S}$$
and the associated transition matrix 
$  \mathcal{P}=[p_{ij}], {i,j\in\N_{0}^{2\Lambda -1}},$
 are determined by  the transition probabilities of
$\{N\}_{\N_0}$ as detailed in the 
 following lemma:
 \begin{lem}
   \label{lem:pij}
   Suppose that Assumption~\ref{ass:iid} holds, then 
   \begin{equation}
     \label{eq:8ab}
     \begin{split}
       &p_{00}=q_{00}, \; p_{10}=q_{10},\; p_{(\Lambda+1)0}=q_{00},\\
       &p_{ij}=q_{ij},\quad \forall (i,j)\in\N_0^\Lambda\times\N_1^\Lambda,\\
       &p_{(j+1)(\Lambda+j)}=q_{(j+1)0},\quad \forall j \in \N_1^{\Lambda-1},\\
       &p_{(\Lambda+m)(\Lambda+m-1)}=q_{00},\quad \forall m\in\N_{2}^{\Lambda-1},\\
       &p_{(\Lambda+k)l}=q_{0l}, \quad \forall (k,l)\in\N_1^{\Lambda-1}\times\N_1^\Lambda.
     \end{split}
   \end{equation}
 All other
transition probabilities in $\mathcal{P}$ are
identically zero. 
 \end{lem}
 \begin{proof}
 See Appendix~\ref{sec:proof-lemma}.  
 \end{proof}

\begin{ex}
Suppose that $\Lambda =3$. Then $\mathbb{S}=\{s_0,\dots,s_5\}$, where
$s_0=(0,0)$, $s_1=(1,1)$, $s_2=(2,2)$,
$s_3=(3,3)$, $s_4=(0,1)$, and $s_5=(0,2)$. The result~(\ref{eq:8ab}) then gives:
\begin{equation*}
\mathcal{P}=
\begin{bmatrix}
   q_{00}&q_{01}&q_{02}&q_{03}& 0&0\\
   q_{10}&q_{11}&q_{12}&q_{13}& 0&0\\
   0&q_{21}&q_{22}&q_{23}&q_{20}&0\\
   0&q_{31}&q_{32}&q_{33}&0&q_{30}\\
   q_{00}&q_{01}&q_{02}&q_{03}& 0&0\\
   0&q_{01}&q_{02}&q_{03}&q_{00}&0
\end{bmatrix}.
\end{equation*}
\end{ex}

\subsection{Distribution of the first return time}
Denote the times when $b(k)$ runs out of calculated control
values, 
i.e., when $Z(k)=s_0=(0,0)$ (equivalently, $\lambda(k)=0$),   via
$\mathcal{K}=\{k_i\}_{i\in\N_0}$, where 
$k_{0}=0$ (from Assumption~\ref{ass:iid}) and
$$  k_{i+1} = \inf \big\{ k\in\N \colon k>k_i,\quad  Z(k)=s_0\big\}, i\in\N_0.$$
We also describe the amount of time steps between  consecutive
elements  of $\mathcal{K}$ via $\Delta_i\in\N$, where:
$$ \Delta_i\eq k_{i+1}-k_i,\quad \forall (k_{i+1}, k_i) \in\mathcal{K} \times\mathcal{K}.$$
Thus, the process $\{\Delta_i\}_{i\in\N_0}$
corresponds to the
first return time of state $s_0$ and is therefore 
i.i.d. (see, e.g., \cite{kemsne60}). Now the 
transition matrix of 
$\{Z\}_{\N_0}$  can be partitioned
according to (see~(\ref{eq:8ab}))
\begin{equation}
\label{eq:28}
  \begin{split}
    \mathcal{P}&=
  \begin{bmatrix}
    q_{00} & \theta^T\\
    \mu &\overline{\mathcal{P}}
  \end{bmatrix},\quad\theta^T =
    \begin{bmatrix}
      q_{01}&\dots&q_{0\Lambda}&(\mathbf{0}_{\Lambda-1})^T
    \end{bmatrix},\quad 
    \overline{\mathcal{P}}=[p_{ij}], \quad i,j\in\N_{1}^{2\Lambda -1}\\
    \mu^T&=
    \begin{cases}
    \begin{bmatrix}
      q_{10}&(\mathbf{0}_{\Lambda-1})^T&q_{00}&(\mathbf{0}_{\Lambda-2})^T
    \end{bmatrix},&\text{if $\Lambda>2$}\vspace{1mm}\\
    \begin{bmatrix}
      q_{10}&0&q_{00}
    \end{bmatrix},&\text{if $\Lambda=2$.}
  \end{cases}
\end{split}
\end{equation}
Lemma~\ref{lem:L2} as proven in Appendix~\ref{sec:proof-lemma-1} characterizes the
distribution of $\{\Delta_i\}$.
\begin{lem}
\label{lem:L2}
 Suppose that Assumption~\ref{ass:iid} holds and consider  $\theta$, $\mu$ and
 $\overline{\mathcal{P}}$ as defined in~\eqref{eq:28}.
Then 
  \begin{equation}
    \label{eq:10}
     \Prob\{\Delta_i=j\}
     =
     \begin{cases}
       q_{00}&\text{if $j=1$,}\\
       \theta^T
       (\overline{\mathcal{P}})^{j-2}
       \mu&\text{if $j\geq 2$.}
     \end{cases}
   \end{equation}
   \end{lem}

\begin{figure}[t]
	\centering
	\includegraphics[width=.6\textwidth]{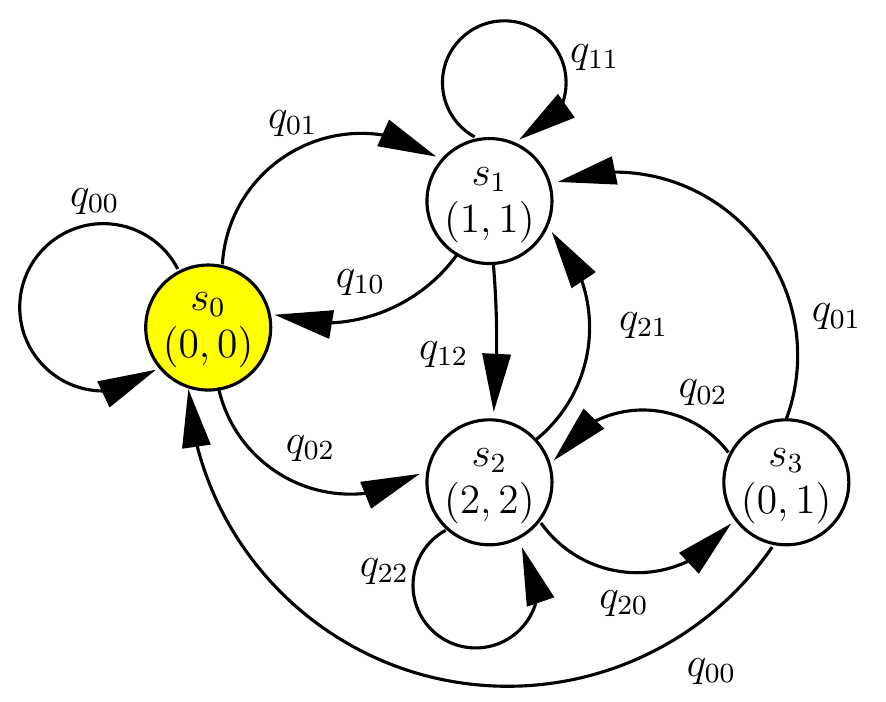}
	\caption{Transition graph of $Z=(N,\lambda)$ for $\Lambda =2$.}
	\label{fig:transitions}
\end{figure}

\begin{ex}
For
$\Lambda = 2$,~(\ref{eq:8ab}) provides the transition matrix
\begin{equation*}
 \mathcal{P}=
 \begin{bmatrix}
   p_{00}&p_{01}&p_{02}&p_{03}\\
   p_{10}& p_{11}&p_{12}&p_{13}\\
   p_{20} & p_{21}&p_{22}&p_{23}\\
   p_{30} & p_{31}&p_{32}&p_{33}
 \end{bmatrix}
 =
\begin{bmatrix}
   q_{00}&q_{01}&q_{02}&0\\
   q_{10}& q_{11}&q_{12}&0\\
   0 & q_{21}&q_{22}&q_{20}\\
   q_{00} & q_{01}&q_{02}&0
 \end{bmatrix}.
\end{equation*}
Thus, for all $j\geq 2$, the result in~(\ref{eq:10}) amounts to:
 \begin{equation*}
   \label{eq:10b}
    \Prob\{\Delta_i=j\}
    =\begin{bmatrix}
   q_{01}&q_{02}&0
 \end{bmatrix}
\begin{bmatrix}
     q_{11}&q_{12}&0\\
     q_{21}&q_{22}&q_{20}\\
     q_{01}&q_{02}&0
   \end{bmatrix}^{j-2}
   \begin{bmatrix}
     q_{10}\\0\\q_{00}
   \end{bmatrix}.
 \end{equation*}
Particular cases of the above  can be visualized by inspecting the graph in
Fig.~\ref{fig:transitions} as follows: The first return times $\{\Delta_i\}$ correspond to
cycles in which $s_0=(0,0)$ is the  starting and ending vertex, but not
otherwise 
contained along the path. Thus, for
$\Delta_i=2$, we have a unique cycle. It has vertices $\{s_0,s_1,s_0\}$, which gives
$\Prob\{\Delta_i=2\}=q_{01}q_{10}$. For $\Delta_i=3$ there are  
three cycles, namely
$\{s_0,s_1,s_1,s_0\}$, $\{s_0,s_2,s_1,s_0\}$, and
$\{s_0,s_2,s_3,s_0\}$. Consequently, we have $\Prob\{\Delta_i=3\} =
q_{01}q_{11}q_{10}+q_{02}q_{21}q_{10}+q_{02}q_{20}q_{00}$.\hfs
\end{ex}

\section{Stability Analysis}
\label{sec:analysis}
Since the processor availability is stochastic, the controller is random,
see~(\ref{eq:4}) and~(\ref{eq:1}). In
particular, if $N(k)=0$ then the plant evolves in open-loop at time $k$
(possibly using tentative plant inputs calculated at previous time-steps); if $\lambda(k)=0$, then the plant input is set to zero at that time.  
\par Various stability notions for
stochastic systems have been studied in the literature; see, e.g.,~\cite{jcfl91,k71}. We focus on the
following: 
\begin{defn}
  A dynamical system with state trajectory $\{x\}_{\N_0}$ is stochastically
  stable, if for some $\varphi \in\mathscr{K}_\infty$, the expected value $\sum_{k={0}}^{\infty}\E\big\{\varphi(|x(k)|)\big\} <\infty.$\hfs
\end{defn}


 Assumption~\ref{ass:bound_prob} stated below, bounds the rate of increase
of  $V$ in~(\ref{eq:3}), when~\eqref{eq:15} is run with zero input. It also imposes a (mild) restriction on
  the distribution of the  initial plant state.
\begin{ass}
  \label{ass:bound_prob}
  There exists $\alpha\in\R_{\geq 0}$
such that  
  \begin{equation}
    \label{eq:20}
     V({f}(x,\mathbf{0}_p))\leq\alpha V(x),\quad\forall x \in\R^n,
   \end{equation}
   and
   $\E\big\{\varphi_2(|x(0)|)\big\}<\infty$, where
$\varphi_2\in\mathscr{K}_\infty$ is as in~(\ref{eq:3}).\hfs
\end{ass}
It is worth noting that, since we allow for $\alpha >1$,
Assumption~\ref{ass:bound_prob} does not require that 
  the open-loop system $x(k+1) = {f}(x(k),\mathbf{0}_p)$  be globally
  asymptotically stable. Further discussion on potential  conservatism imposed by these
  assumptions can be found in Section IV-A of \cite{quegup12a}.  

\subsection{Stability with Algorithm A$_1$}
\label{sec:algorithm-a_1}
To study stochastic stability when Algorithm
$\textsc{A}_{1}$ is used, we will focus on the random instances where the buffer runs out of control inputs.
\begin{lem}
\label{lem:Markov}
  With Algorithm A$_1$, the plant state sequence \emph{at the time steps 
    $k_i\in\K$}, namely $\{x\}_{\K}$, 
is  Markovian.\hfs
\end{lem}
\begin{proof}
  It follows from the definition of $k_{i}$ that $\forall k_i\in\K$ we have
  $  u(k_i) =\mathbf{0}_p,$
  $ b(k_i)=\mathbf{0}_{\Lambda p}$, 
  $\lambda(k_i) =  N(k_i)=0$. Thus, the plant state at
  time $k_{i+1}$  depends only on $x(k_i)$ and
  $\{N(k_i+1),N(k_i+2),\dots,N(k_{i+1}-1)\}$. The result follows
 from the Markovian property of $\{N\}_{\N_0}$. 
\end{proof}

Based on the results of Section~\ref{sec:analys-via-assoc} and
Lemma~\ref{lem:Markov}, stochastic  stability of the control system can  be 
analyzed by using a stochastic Lyapunov function approach as follows:
\begin{lem}
  \label{lemma:anytime_inter}
  Suppose that Assumptions~\ref{ass:CLF} to~\ref{ass:bound_prob} hold and
  consider $k_0,k_1\in \K$. We then have
  \begin{equation}
    \label{eq:6}
    \E\big\{V(x({k_{1}}))\,\big|\,x(k_{0})=\chi\big\}\\
 \leq \Omega V(\chi),\quad 
 \forall \chi\in \R^n,
\end{equation}
where\footnote{Note that, since $\rho\in[0,1)$, $\Omega$ is bounded.} 
\begin{equation}
\label{eq:13}
  \Omega\eq\alpha \sum_{j\in\N} \Prob\{ \Delta_i =j\} \rho^{j -1}.
\end{equation}
\end{lem}

\begin{proof}
By Lemma~\ref{lem:Markov} and Assumptions~\ref{ass:CLF}
and~\ref{ass:bound_prob}, we have 
\begin{equation}
\label{eq:2}
  \begin{split}
    V(x(k_0+1))&\leq \alpha  V(x(k_0)),\\
    V(x(k+1))&\leq \rho  V(x(k)),\quad \forall k \not \in \K.
  \end{split}
\end{equation}
Thus, $\E\big\{V(x({k_{1}}))\,\big|\,x(k_{0})=\chi,\Delta_0=j\big\}
 \leq \alpha\rho^{j -1} V(\chi)$,
for all $\chi\in \R^n$. The
result~\eqref{eq:6} follows by using the law of total expectation and the fact
that $\{\Delta_i\}$ is i.i.d.  
\end{proof}

Although Lemma~\ref{lemma:anytime_inter} considers only the
instants $k_0$ and $k_1$, the bound in~\eqref{eq:6}
can be used to conclude about stochastic 
stability  for all
$k\in\N_0$  
\begin{theorem}
  \label{theorem:a1_stability}
  Suppose that
  Assumptions~\ref{ass:CLF}--\ref{ass:bound_prob} hold and that 
$\Omega  <1$. Then the plant state trajectory when controlled with Algorithm
A$_1$ is stochastically stable with the bound:
\begin{equation*}
    \sum_{k={0}}^{\infty}\E\big\{\varphi_1(|x(k)|)\big\} 
    <\frac{1+\alpha-\rho}{(1-\Omega)(1-\rho)}
    \E\big\{\varphi_2(|x(0)|)\big\}.
  \end{equation*}
\end{theorem}
\begin{proof}
From~(\ref{eq:3}) and Lemmas~\ref{lem:Markov} and~\ref{lemma:anytime_inter}, it
follows that if $\Omega  <1$, then $V(x({k_{i}}))$ is a
  stochastic Lyapunov function for  $\{x\}_{\K}$. Therefore,~\cite[Chapter
  8.4.2, Theorem 2]{k71} implies exponential stability at instants
  $k_{i}\in\mathcal{K}$, i.e., for all
    $(i,\chi_0)\in\N_0\times\R^n$,
  $$  \E\{V(x({k_{i}}))\,|\,x(k_{0})=\chi_0\}
  \leq\Omega^{i}V(\chi_0).$$
    For the time steps $k\in \N \,\backslash \, \mathcal{K}$, i.e., where
   calculated control values are applied,~(\ref{eq:2}) gives
  \begin{equation*}
    \begin{split}
 \E&\Bigg\{\sum_{k=k_{i}}^{k_{i+1}-1}V(x({k}))\,\bigg|\,x(k_{i})=\chi_i,\Delta_i=j\geq
 2\Bigg\}\leq \Bigg(1+ \alpha\sum_{l=0}^{j-2} \rho^l\Bigg) V(\chi_i)
 \leq \bigg(1
 +\frac{\alpha}{1-\rho}\bigg)V(\chi_i).
\end{split}
\end{equation*}
The latter bound holds for all $j\geq 2$. Now, using the law of total
expectation, we obtain
$$  \E\Bigg\{\sum_{k=k_{i}}^{k_{i+1}-1}V(x({k}))\,\bigg|\,x(k_{i})=\chi_i\Bigg\}\leq 
 \frac{1+\alpha-\rho}{1-\rho}V(\chi_i).$$
Taking conditional expectation $\E\{\, \cdot\,|\, x(k_0)=\chi_0\}$ on both sides, defining $\beta \eq  (1+\alpha-\rho)/(1-\rho)$  and using the Markovian property of $\{x\}_{\K}$ yields
  \begin{equation*}
\begin{split}
 & \E\Bigg\{ \E\Bigg\{\sum_{k=k_{i}}^{k_{i+1}-1}V(x({k}))\,\bigg|\, x(k_{i})=\chi_i\Bigg\}
    \,\Bigg|\, x(k_0)=\chi_0\Bigg\} \\
    &= \E\Bigg\{\sum_{k=k_{i}}^{k_{i+1}-1}V(x({k}))\,\bigg|\,
    x(k_{0})=\chi_0\Bigg\}
    \leq   \beta \E\big\{V(x({k_{i}}))\,\big|\,x(k_0)=\chi_0\big\}
    \leq \beta \Omega^iV(\chi_0).
  \end{split}
\end{equation*}
Thus,
 $$   \E\Bigg\{\sum_{k={k_0}}^{k_{j+1}-1}V(x({k}))\,\bigg|x(k_0)=\chi_0\Bigg\}
 \leq \beta\sum_{i=0}^j\Omega^iV(\chi_0).$$
Now let  $k_{j+1}\to\infty$ and recall that $k_0=0$ to obtain
$$\sum_{k=0}^{\infty}
  \E\big\{V(x(k))\,\big|\,x(0)=\chi_0\big\}\leq
  \frac{\beta}{1-\Omega} 
  V(\chi_0).$$
The result now follows by using~(\ref{eq:3}), Assumption~\ref{ass:bound_prob}
and taking 
expectation with respect to the distribution of $x(0)$.
\end{proof}

Theorem~\ref{theorem:a1_stability} establishes sufficient conditions for
stochastic stability of the control loop when Algorithm A$_1$ is used and
processor availability is Markovian.  The quantity $\Omega$ involves
 the contraction factor of the baseline controller
$\kappa$,
see~(\ref{eq:3}), the bound on 
the rate of increase of $V$  when the plant input is zero,
see~(\ref{eq:20}), and the  
distribution of $\{\Delta_i\}$, $i\in\N_0$ which was characterised in Lemma~\ref{lem:L2}.
In Section~\ref{sec:relat-prev-results}, we will
relate Theorem~\ref{theorem:a1_stability} to the relevant result
in \cite{quegup12a}.   Before doing so, we will first investigate the baseline 
algorithm. 

\subsection{Stability with the Baseline Algorithm}
\label{sec:stab-with-basel}
Sufficient conditions for stochastic stability  when the
baseline algorithm 
in~(\ref{eq:4}) is used can be established by proceeding in a similar manner as
was done for Algorithm A$_1$.  
Here, we note that the baseline controller is
characterised via:
\begin{equation}
  \label{eq:4c}
  u(k)=
  \begin{cases}
    \kappa(x(k)) &\text{if $N(k)\in \N_1^{\Lambda}$,}\\
    \mathbf{0}_p &\text{if $N(k)=0$}.
  \end{cases}
\end{equation}
Denote the time steps
where $N(k)=0$ as $\mathcal{T}=\{t_i\}$ 
where
$$  t_{i+1} = \inf \big\{ k\in\N \colon k>t_i,\quad  N(k)=0\big\}, i\in\N_0,$$
 with $t_0=0$. Further, we introduce the process
$\{\tau_i\}_{i\in\N_0}$ consisting of the times between
consecutive elements of $\mathcal{T}$ via the relation
$$  \tau_i=t_{i+1}-t_i\in\N,\quad (t_i,t_{i+1})\in\mathcal{T}\times\mathcal{T}.$$
Thus, $\{\tau_i\}_{i\in\N_0}$
  are the first return times to state $0$ of the Markov Chain
  $\{N\}_{\N_0}$, and are therefore i.i.d.  Fig.~\ref{fig:transitions2} can be used to
visualize  $\{\tau_i\}_{i\in\N_0}$ for the case  $\Lambda=2$. This should be
contrasted with how $\{\Delta_i\}_{i\in\N_0}$ is illustrated in Fig.~\ref{fig:transitions}. 
 
\par By adapting the proof of Lemma~\ref{lem:L2}, we can characterize the distribution of $\{\tau_i\}_{i\in\N_0}$ as follows:
\begin{lem}
\label{lem:stab-with-base}
 Suppose that Assumption~\ref{ass:iid} holds. Then,
 $\Prob\{\tau_i=1\}=q_{00}$ and, for $j\geq 2$,  
  \begin{equation*}
      \Prob\{\tau_i=j\}
     =
     \begin{bmatrix}
       q_{01}&\dots&q_{0\Lambda} 
     \end{bmatrix}
     \begin{bmatrix}
    q_{11}&\dots&q_{1\Lambda}\\
    \vdots&\ddots&\vdots\\
    q_{\Lambda 1} &\dots&q_{\Lambda\Lambda}
  \end{bmatrix}^{j-2}
  \begin{bmatrix}
    q_{10}\\ \vdots \\ q_{\Lambda 0}
  \end{bmatrix}\!.
  \end{equation*}
\end{lem}

\begin{theorem}
  \label{thm:baseline}
Suppose that
  Assumptions~\ref{ass:CLF}--\ref{ass:bound_prob} hold and that the baseline
  algorithm in~(\ref{eq:4c}) is used. If
  \begin{equation}
  \label{eq:85b}
\Theta\eq \alpha \sum_{j\in\N} \Prob\{ \tau_i =j\} \rho^{j -1}  <1, 
\end{equation}
 then  the control loop is stochastically stable. In particular, 
\begin{equation*}
    \label{eq:33b}
    \sum_{k={0}}^{\infty}\E\big\{\varphi_1(|x(k)|)\big\} 
    <\frac{1+\alpha-\rho}{(1-\Theta)(1-\rho)}
    \E\big\{\varphi_2(|x(0)|)\big\}. 
  \end{equation*}
\end{theorem} 

\begin{proof}
  By
  adapting the above ideas, 
  it can be
  shown that $V(x(t_i))$ is a stochastic Lyapunov function for the Markov process
  $\{x\}_{\mathcal{T}}$. The remainder of the proof
   then parallels that of Theorem~\ref{theorem:a1_stability}, but using
   $\{\tau_i\}_{i\in\N_0}$ instead of
   $\{\Delta_i\}_{i\in\N_0}$.
\end{proof}

\begin{ex}
 With $\Lambda=2$, Lemma~\ref{lem:stab-with-base} gives
  that
  \begin{equation*}
    \Prob\{\tau_i=j\}=
    \begin{cases}
      q_{00}&\text{if $j=1$,}\\
      q_{01} q_{11} ^{j-2} q_{10} &\text{if $j\geq 2$.}
    \end{cases}
  \end{equation*}
  in which case the sufficient condition~(\ref{eq:85b}) reduces to
  \begin{equation*}
    \begin{split}
      \Theta&=\alpha q_{00} + (1-q_{00})(1-q_{11}) \alpha \rho \sum_{j\geq 2}
      q_{11}
      ^{j-2} \rho^{j-2}=\alpha q_{00} + \alpha \rho\frac{(1-q_{00})(1-q_{11})}{1-\rho q_{11}}\\
      &=\frac{\alpha q_{00}(1-\rho q_{11})+\alpha \rho
        (1-q_{00})(1-q_{11})}{1-\rho q_{11}}=\frac{\alpha q_{00}(1-\rho q_{11})+\alpha \rho
        (1-q_{00})(1-q_{11})}{1-\rho q_{11}} <1.
    \end{split}
  \end{equation*}
\end{ex}

\section{Relationship to Previous Stability Results }
\label{sec:relat-prev-results}
In Section VII of\cite{quegup12a} we examined Algorithm A$_1$   using a model 
for the processor availability that allows for correlations in
$\{N\}_{\N_0}$, by introducing a \emph{processor state
process} $\{g\}_{\N_0}$ with values in $\N_1^G$,
  $G\in\N.$ The process is described by an irreducible aperiodic Markov Chain
  with transition  matrix $Q'=\left[q'_{ij}\right]_{G\times G}$
with 
$$ q'_{ij} = \Prob\{ g(k+1) = j\,|\, g(k) = i\}.$$ The realizations of
  $\{g\}_{\N_0}$ determine $\{N\}_{\N_0}$ 
 as per
\begin{equation}
  \label{eq:40b}
  \begin{split}
    \Prob&\{N(k)=l\,|\, g(k)=\varsigma\} = p'_{l|\varsigma},\quad \forall
    (l,\varsigma)\in\N_0^\Lambda \times \N_1^G,
  \end{split}
\end{equation}
with given probabilities $p'_{l|\varsigma}$. Clearly, our model in Assumption~\ref{ass:iid} can be described using this
structure by setting the processor state to satisfy
$  g(k)=N(k)+1, \forall k\in\N_0,$
in which case $G=\Lambda+1$, and $p'_{l|\varsigma}=1$ if  $\varsigma=l+1  $ and is 0 otherwise.
In particular, we have $p'_{0|\varsigma}\in\{0,1\}$ with
  $p'_{0|\varsigma}=1$ if and only if $\varsigma = 1$.
Note that $Q' = \mathcal{Q}$ since 
\begin{equation*}
  q'_{ij} = \Prob\{ g(k+1) = j\,|\, g(k) = i\}=\Prob\{ N(k+1) +1 = j\,|\,
  N(k)+1   =   i\} = q_{i-1,j-1}.
\end{equation*}
 
\subsection{The Baseline Algorithm}
\label{sec:baseline-algorithm-1}
 Given the above, for the model in Assumption~\ref{ass:iid}, Theorem 4 of
\cite{quegup12a} establishes that if  $\alpha<1$, then the closed loop
system when using the baseline algorithm   is stochastically
stable. This is in contrast to  the results in our
current work, which also allow for $\alpha >1$. Thus,  Theorem~\ref{thm:baseline} 
provides a sufficient condition which can also be used for open-loop unstable
plant models. This follows directly from
Lemma~\ref{lem:stab-with-base} and by the fact that $\rho \in[0,1)$, so that
\begin{equation*}
  \sum_{j\in\N} \Prob\{ \tau_i =j\} \rho^{j -1} =   q_{00} + 
    \sum_{j= 2}^{\infty} \Prob\{ \tau_i =j\} \rho^{j -1}
    \leq   q_{00} + \rho 
    \sum_{j\geq 2} \Prob\{ \tau_i =j\} = q_{00} + \rho (1-q_{00})<1.
 \end{equation*}

\subsection{The Anytime Algorithm}
Application of Theorem 5
of\cite{quegup12a} to our present model gives the following sufficient condition
for stochastic stability, that is proven in Appendix~\ref{sec:proof-corollary}.

\begin{coro}\label{c14}
Let
  Assumptions~\ref{ass:CLF}--\ref{ass:bound_prob} hold. Suppose that $\alpha
  q_{00}<1$ and that $\Upsilon_\varsigma <1,\forall \varsigma
  \in \N_2^{\Lambda+1},$ where
  \begin{equation}
   \label{eq:9}
   \Upsilon_\varsigma  \triangleq\frac{q_{(\varsigma-1)0}(1- q_{00})}{1-q_{00}\rho}\Bigg(
   \frac{(\alpha-\rho)}{1-q_{00}\alpha}q_{00}^{\varsigma-2}\rho^{\varsigma-1}+\rho^{2}\Bigg)
   +\rho (1-q_{(\varsigma-1)0}).
 \end{equation}
Then the plant state trajectory when controlled with Algorithm
 A$_1$ is stochastically stable.\hfs
\end{coro}

In general, comparing the above sufficient condition with the one presented in 
Theorem~\ref{theorem:a1_stability} is difficult. To elucidate the situation, in
the remainder of this section we will focus on processor availability models
where all transition probabilities are equal, i.e.,
\begin{equation}
  \label{eq:45}
  q_{ij}=1/(\Lambda+1),\quad  \forall i,j\in\N_0^\Lambda.
\end{equation}
For this class of models,  the result in~(\ref{eq:10}) can be written as:
\begin{equation}
  \label{eq:332}
  \Prob\{\Delta_i=j\}
     =(1/(\Lambda+1))^{j}\sum_{k=1}^{\Lambda}\sum_{l=1}^{\Lambda}v_{kl}\lambda_k^{j-2},\quad j\geq 2,
\end{equation} 
where $\{\lambda_k\}$, $k\in\N_1^\Lambda$, are the non-zero eigenvalues of
$\underline{\mathcal{P}}\eq(1/(\Lambda+1))^{2-j}\overline{P}$, $\{v_{k}\}$  are the
corresponding eigenvectors such that $\mu^Tv_k\neq 0$, and $v_{kl}$ denotes the $l$-th element
of $v_k$.  
Expressions~\eqref{eq:13} and~(\ref{eq:332}) yield
\begin{equation*}
  \begin{split}
    \Omega 
    &= (\alpha/(\Lambda+1)) + \alpha\rho \sum_{j\geq
      2}(1/(\Lambda+1))^{j}\sum_{k=1}^{\Lambda}\sum_{l=1}^{\Lambda}v_{kl}(\lambda_k\rho)^{j-2}
    = (\alpha/(\Lambda+1)) \bar f_\Lambda (\rho),
  \end{split}
\end{equation*}
where\footnote{Notice that the first row of $\underline{\mathcal{P}}$ is
  $[I_{1\times \Lambda}~\mathbf{0}_{\Lambda-1}^T]$, hence
  $\sum_{l=1}^{\Lambda}v_{kl}=\lambda_kv_{k1}$. From (\ref{eq:8ab}) and
  Gershgorin circle theorem, we know that $|\lambda_k|\leq \Lambda+1$. Since $|\rho|
  <1$, we have $|\lambda_k\rho/(\Lambda+1)|<1$, $\forall k\in\N_{1}^\Lambda$.}   
\begin{equation*}
  \label{eq:42}
  \bar f_\Lambda (\rho) \eq
  1+\frac{\rho}{\Lambda+1}\Bigg(\sum_{k=1}^{\Lambda}\frac{1}{1-\lambda_k\rho/(\Lambda+1)}\sum_{l=1}^{\Lambda}v_{kl}\Bigg)
  = 1+ \rho \sum_{k=1}^{\Lambda}\frac{\lambda_kv_{k1}}{\Lambda+1-\lambda_k\rho}. 
\end{equation*}
Therefore, $\alpha \leq
({\Lambda+1})/{\bar{f}_\Lambda(\rho)}\Leftrightarrow\Omega \leq 1$.

\par On the other hand, with $q_{00}\alpha<1$,~\eqref{eq:9} yields
\begin{equation*}
  \begin{split}
    \Upsilon_2&= \frac{\rho q_{10}(1- q_{00})}{1-q_{00}\rho}\Bigg(
    \frac{(\alpha-\rho)}{1-q_{00}\alpha} +\rho\Bigg) +\rho (1-q_{10})
    =\rho+\frac{\rho (\alpha-1) q_{10} }{1-q_{00}\alpha} 
=\frac{\rho\Lambda}{\Lambda+1-\alpha},\quad
    \forall \Lambda \geq 2.
  \end{split}
\end{equation*}
Thus $\alpha \leq \Lambda+1-\Lambda\rho \Leftrightarrow \Upsilon_2 \leq 1
$. Recalling that $\alpha q_{00}<1$, $  q_{ij}=1/(\Lambda+1),~  \forall i,j\in\N_0^\Lambda$ and
$\rho<1$, it follows that $ \Upsilon_2\geq \Upsilon_\varsigma$, for all $\varsigma\in\N_2^{\Lambda+1}$.

The above analysis leads to the following characterization on cases where the stability condition developed in this work is less conservative than the one presented in  \cite{quegup12a}.

\begin{coro}\label{c15}
Consider a processor availability model of the form~(\ref{eq:45})  and suppose that
$\alpha<\Lambda+1$. If $${\bar{f}_\Lambda(\rho)}( \Lambda+1-\Lambda\rho)  <
\Lambda+1,$$ then the
sufficient condition for stability in Theorem 5 of\cite{quegup12a} is more
conservative than the one derived in Theorem~\ref{theorem:a1_stability} of the present work. \hfs
\end{coro}
 Additional comparisons are provided in Section~\ref{sec:case-studies}.

\begin{ex}
\label{ex:better} For $\Lambda=2$,  the result in~(\ref{eq:10}) amounts to:
\begin{equation}
  \label{eq:31}
  \Prob\{\Delta_i=j\}
     =(1/3)^{j}
     \begin{bmatrix}
    1&1&0
  \end{bmatrix}
 \begin{bmatrix}
      1&1&0\\
      1&1&1\\
      1&1&0
    \end{bmatrix}^{j-2}
    \begin{bmatrix}
      1\\0\\1
    \end{bmatrix},\quad j\geq 2.
\end{equation}
By decomposing   $[1\;0\;1]^T$ into the eigenvectors of the matrix above, we obtain
\begin{equation*}
  \begin{split}
     &\begin{bmatrix}
    1&1&0
  \end{bmatrix}\begin{bmatrix}
      1&1&0\\
      1&1&1\\
      1&1&0
    \end{bmatrix}^{j-2}
    \begin{bmatrix}
      1\\0\\1
    \end{bmatrix} =\begin{bmatrix}
    1&1&0
  \end{bmatrix} \begin{bmatrix}
      1&1&0\\
      1&1&1\\
      1&1&0
    \end{bmatrix}^{j-2}
    \left(\begin{bmatrix}
      1/2\\\sqrt{2}/2\\1/2
    \end{bmatrix}+\begin{bmatrix}
      1/2\\ -\sqrt{2}/2 \\1/2
    \end{bmatrix} \right)\\
  &=\begin{bmatrix}
    1&1&0
  \end{bmatrix}(1+\sqrt{2})^{j-2}
    \begin{bmatrix}
      1/2\\\sqrt{2}/2\\1/2
    \end{bmatrix}
    +\begin{bmatrix}
    1&1&0
  \end{bmatrix} (1-\sqrt{2})^{j-2} 
    \begin{bmatrix}
      1/2\\ -\sqrt{2}/2 \\1/2
    \end{bmatrix},
  \end{split}
\end{equation*}
so that
\begin{equation}
\label{eq:34}
\Prob\{\Delta_i=j\}
     =(1/3)^{j}
  \frac{(1+\sqrt{2})^{j-1} + (1-\sqrt{2})^{j-1} }{2},\quad j\geq 2.
\end{equation}
Thus (and after some algebraic manipulations),~\eqref{eq:13} yields 
\begin{equation}
  \label{eq:35}
  \begin{split}
    \Omega 
    &= (\alpha/3) +\frac{\alpha}{2} \sum_{j\geq
      2}(1/3)^{j}(1+\sqrt{2})^{j-1} \rho^{j-1} + \frac{\alpha}{2} \sum_{j\geq
      2}(1/3)^{j}(1-\sqrt{2})^{j-1} \rho^{j-1}
    = (\alpha/3)\bar f_2 (\rho).
  \end{split}
\end{equation}
where
\begin{equation}
  \label{eq:41}
  \bar f_2 (\rho)\eq 1 + (\rho/2)\Bigg( 
  \frac{(1+\sqrt{2})(3- \rho(1-\sqrt{2})) + (1-\sqrt{2})(3- \rho(1+\sqrt{2}))}{(3- \rho(1+\sqrt{2}))(3- \rho(1-\sqrt{2}))}\Bigg)=\frac{3(3-\rho)}{9-6\rho-\rho^2}.
\end{equation}
Using Corollary~\ref{c15} and by noting that
\begin{equation}
  \label{eq:50}
  (3-2\rho) \bar{f}_2(\rho)\leq 3 \Longleftrightarrow  \rho(1-\rho) \geq 0,
\end{equation}
we conclude that the upper bound on $\alpha$ permited in Theorem 5 of
\cite{quegup12a} is smaller than the one allowed  in the present work. \hfs    
\end{ex}

\section{Robustness To Process Noise}
\label{sec:notes-rubustness}
Our presentation so far assumed no process noise in (\ref{eq:15}). A natural
question is if the quality of future inputs, and hence the performance of the
algorithm, degrades if process noise is present. We now  consider the case
where the system model is given by
\begin{equation}
\label{eq:51}
x(k+1)=f(x(k),u(k),w(k)),
\end{equation}
where $w(k)\in\R^n$ is a white noise process, assumed independent of the
other random variables in the system. For simplicity we shall assume uniform
continuity and bounds as follows:
\begin{ass}
\label{ass:tool}
 There exist $\lambda_x,\lambda_u,\lambda_w,\lambda_V,\lambda_{\kappa}, \rho,
 \beta, \alpha, \eta\in \R_{\geq 0}$ such that, $\forall x,z,w\in \R^n$ and $\forall  u,v\in \R^p$ the following are satisfied:
\begin{equation}
  \label{eq:24}
  \begin{split}
      |f(x,u,w)-f(z,v,\mathbf{0}_n)|&\leq \lambda_x|x-z|+\lambda_u|u-v|+\lambda_w|w|,\\
    |V(x)-V(z)|&\leq \lambda_V |x-z|,\\
     |\kappa(x)-\kappa(z)|&\leq \lambda_\kappa |x-z|,
  \end{split}
\end{equation}
\begin{equation}
  \label{eq:29}
  \begin{split}
    V(f(x,\kappa(x),w))&\leq \rho V(x) + \beta |w|,\\
    V(f(x,\mathbf{0}_p,w))&\leq \alpha V(x) + \eta |w|.
  \end{split}
\end{equation}
\end{ass}

The following result shows that the condition $\Omega<1$, used in Theorem
\ref{theorem:a1_stability}, plays an important role
also in the present robustness analysis. As in related results on stochastic
stability with unbounded dropouts and disturbances (see, e.g., \cite{quenes12a}), the property established
is weaker than that of  Theorem \ref{theorem:a1_stability}.


\begin{theorem}
  \label{thm:robust}
  Suppose that  
  Assumptions~\ref{ass:CLF}--\ref{ass:tool}   hold, that $\E\{|w(k)|\}<\infty$ and that
$\Omega  <1$. Then the plant state trajectory when controlled with Algorithm
A$_1$ satisfies $\E\{\varphi_1(|x(k)|)\}<\infty$, $ \forall k\in \N_0.$\hfs
\end{theorem}

\section{Case Studies}
\label{sec:case-studies}
\begin{figure}[t] 
\centering
\includegraphics[width=.5\textwidth]{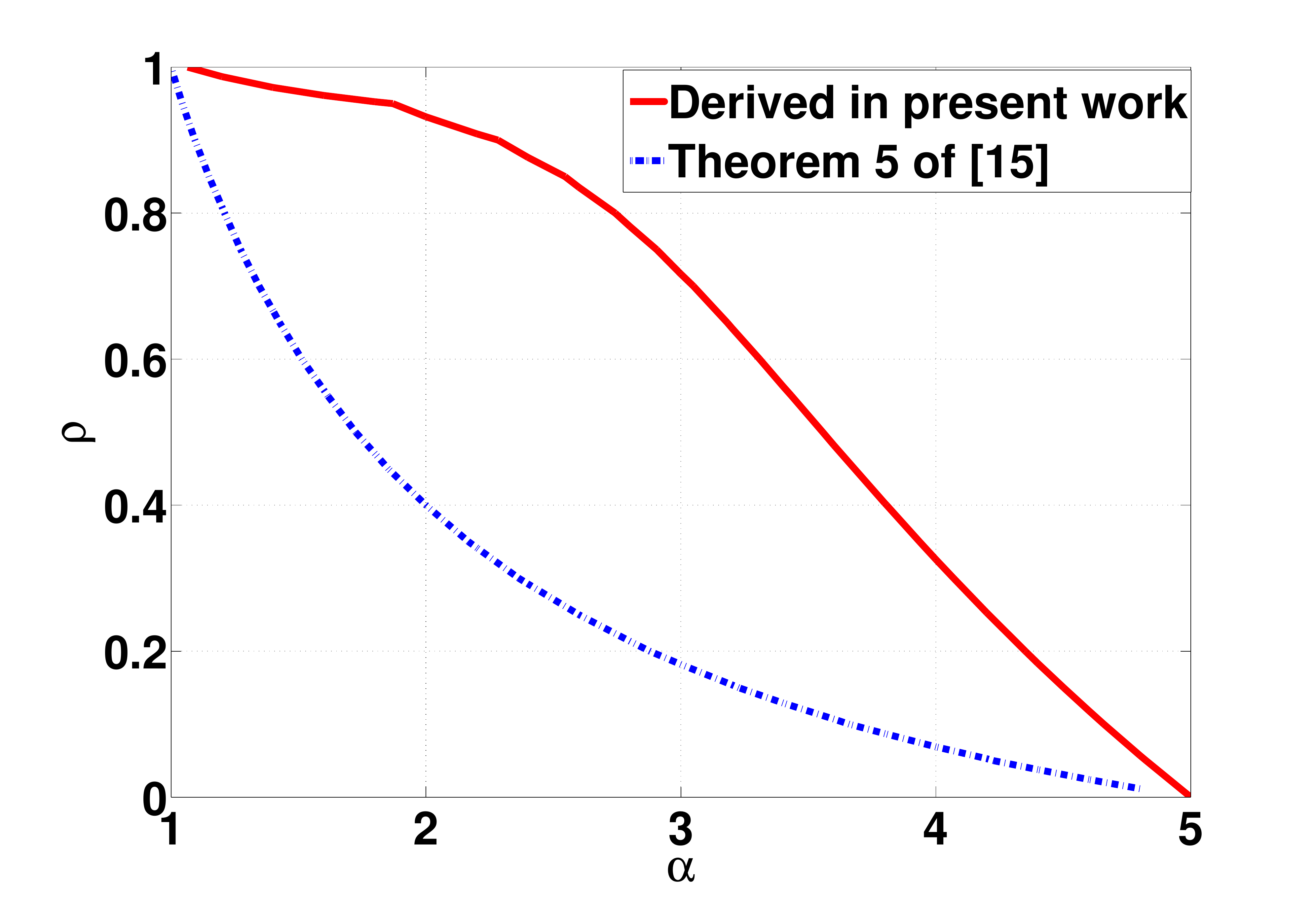}
	\caption{Boundaries of stability regions for the example considered.}
	\label{region}
      \end{figure}

We consider a processor availability model with $\Lambda=5$ and the transition probability matrix $\mathcal{Q}$ with $q_{00}=0.2,~ q_{0j}=0.16,~\forall j\in\N_1^{\Lambda}, ~q_{10}=0.9,~q_{11}=q_{12}=0.05,~q_{21}=0.1,~q_{2j}=0.225,~\forall j\in\N_2^{\Lambda},~q_{ij}=0.25,~\forall i\in\N_3^{\Lambda},~\forall j\in\N_2^{\Lambda}$. All other transition probabilities are identically zero:
\begin{equation*}
\mathcal{Q}=
\begin{bmatrix}
   0.2&0.16&0.16&0.16& 0.16&0.16\\
   0.9&0.05&0.05&0& 0&0\\
   0&0.1&0.225&0.225&0.225&0.225\\
   0&0&0.25&0.25&0.25&0.25\\
   0&0&0.25&0.25& 0.25&0.25\\
   0&0&0.25&0.25&0.25&0.25
\end{bmatrix}.
\end{equation*}
 Intuitively, since the sufficient
condition for stability in Corollary~\ref{c14} is based on a worst case
analysis, whereas the condition in Theorem~\ref{theorem:a1_stability} is not,
the latter result can be expected to be less conservative than the former. This
conjecture was verified in Example~\ref{ex:better}  and is further illustrated  in Fig.~\ref{region} that characterizes the stability region boundaries
in terms of $\alpha$ and $\rho$.  The stable region (area under the curve) as derived from the condition given in Theorem~\ref{theorem:a1_stability} is larger than the one derived from 
Corollary~\ref{c14} (which embodies Theorem 5 of
\cite{quegup12a}).      

 \begin{figure}[t] 
\centering 
	\includegraphics[width=.7\textwidth]{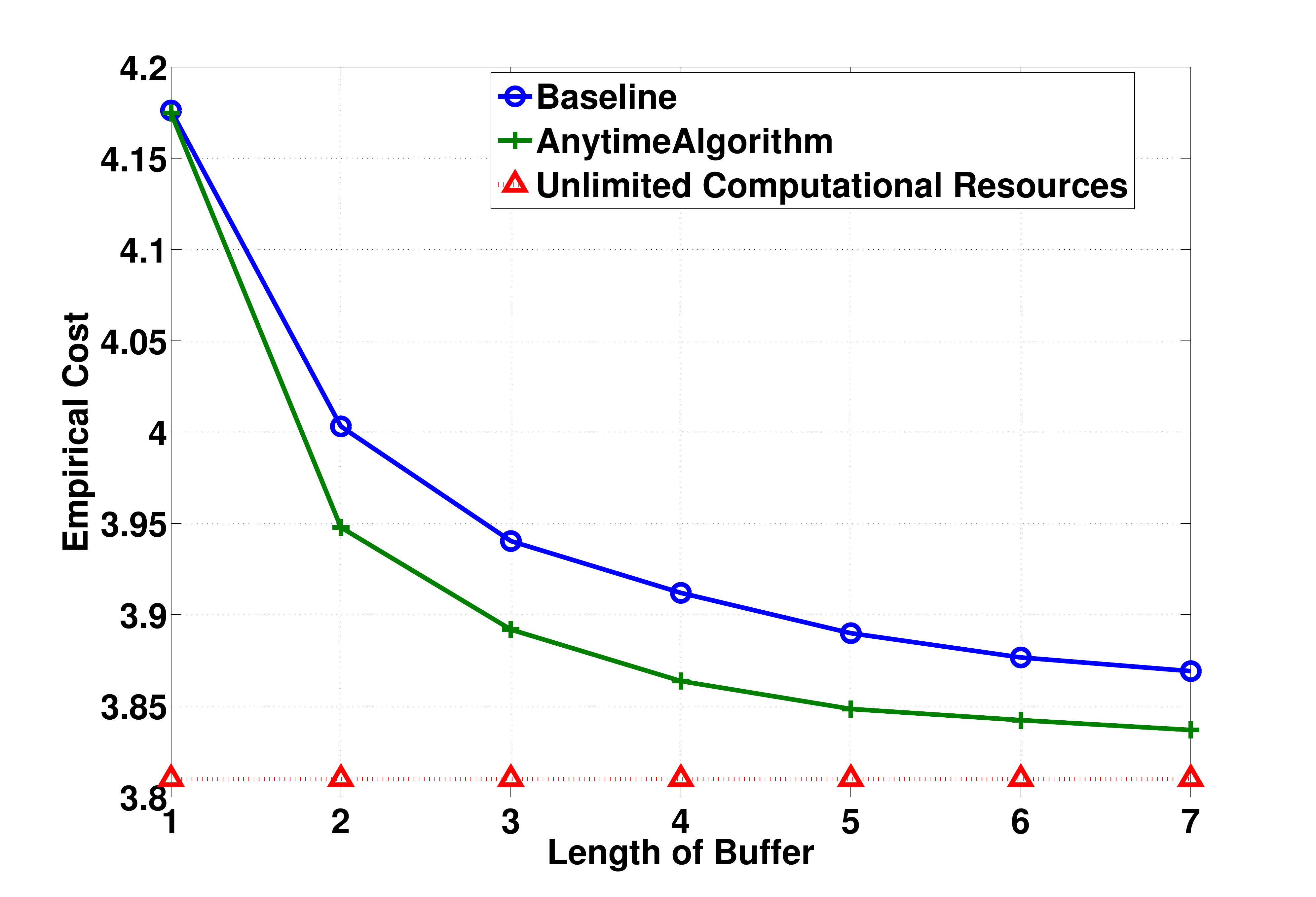}
	\caption{Empirical cost when controlling~(\ref{eq:8}) with the anytime algorithm
  and with the baseline algorithm~(\ref{eq:4c}), as a function of the parameter
  $\Lambda$, see~(\ref{eq:11}). Also included is the ideal case with unlimited
   resources, where $u(k)=\kappa(x(k))$, for all $k\in\N$. }
	\label{comparison_multibit}
\end{figure}

\par Next, consider a specific non-linear plant model of the form~(\ref{eq:15}), where
\begin{equation}
  \label{eq:8}
   x(k+1)=x(k)+0.01((x(k))^3+u(k)).
\end{equation}
A control law satisfying Assumption~\ref{ass:CLF} is given by $\kappa(x)=-x^3-x$,
which is globally stabilizing with $V(x)=|x|$,
$\varphi_1(|x|)=\varphi_2(|x|)=|x|$ and $\rho = 0.99$, see \cite{quegup12a}.
Consider the
following class  of processor availability transition matrices: 
\begin{equation}
\label{eq:11}
 \mathcal{Q}_{\Lambda} =
 \begin{bmatrix}
    0.4&0.6/\Lambda&0.6/\Lambda&\dots\\
    0.6/\Lambda& 0.4&0.6/\Lambda&\dots\\
    0.6/\Lambda&0.6/\Lambda& 0.4&\ddots\\
    \vdots&\vdots&\ddots &\ddots
  \end{bmatrix}\in\R^{(\Lambda+1)\times(\Lambda+1)}.
\end{equation}
In~(\ref{eq:11}), $\Lambda \in\{1,2,\dots,7\}$ is a parameter which determines
the support of $\{N\}_{\N_0}$ and also 
how likely the processor availability changes.
We adopt as  performance measure, the empirical cost
 $$  J=\frac{1}{50}\E\left\{
     \sum_{k=0}^{49}\left(0.2(x(k))^{2}+2(u(k))^{2}
     \right)\right\}, $$
where expectation is taken with respect to the process
$\{N\}_{\N_0}$. Fig.~\ref{comparison_multibit} illustrates the result obtained
when using the anytime algorithm A$_1$ and also the baseline algorithm~(\ref{eq:4c}). The anytime control
algorithm  outperforms the baseline controller
 for  all processor availability models considered. Fig.~\ref{comparison_multibit} also compares the performance with the one without computational
uncertainty, i.e. $N(k)>0,~\forall k$. This comparison characterizes the degradation in performance due to fluctuating CPU time. 

\section{Conclusions}
\label{sec:conclusions}
We  analyzed an anytime control algorithm    when the processor
availability is described by a Markov Chain. The algorithm partially compensates for the effect of  the processor not providing sufficient resources at some time steps. For general non-linear systems, we
used 
stochastic Lyapunov methods to obtain
sufficient conditions for  stability. The results obtained complement those of
our recent article \cite{quegup12a}. In subsequent work, see\cite{quejur13a},  we have
shown how to use  the present analysis methodology  for   networked
control systems with random delays and dropouts.

\appendices
\section{Proof of Lemma~\ref{lem:pij}}
\label{sec:proof-lemma}
If Assumption~\ref{ass:iid} holds,  the definition of  $\mathbb{S}$,~(\ref{eq:19}) and 
the Markovian property of 
$\{N\}_{\N_0}$ yield:
\begin{equation*}
  \begin{split}
  p_{00}&=\Prob\{Z(k+1)=s_0\,|\,Z(k)=s_0\}=\Prob\{N(k+1)=\lambda(k+1)=0\,|\,N(k)=\lambda(k)=0\}\\
  &=\Prob\{N(k+1)=0\,|\,N(k)=0\}=q_{00}\\
  p_{10}&=\Prob\{Z(k+1)=s_0\,|\,Z(k)=s_1\}=\Prob\{N(k+1)=\lambda(k+1)=0\,|\,N(k)=\lambda(k)=1\}\\
  &=\Prob\{N(k+1)=0\,|\,N(k)=1\}=q_{10}\\
  p_{(\Lambda+1)0}&=\Prob\{Z(k+1)=s_0\,|\,Z(k)=s_{\Lambda+1}\}=\Prob\{N(k+1)=\lambda(k+1)=0\,|\,N(k)=0,\lambda(k) =1\}\\
  &=\Prob\{N(k+1)=0\,|\,N(k)=0\}=q_{00}.
\end{split}
\end{equation*}
Similarly, for all $(i,j)    \in\N_0^\Lambda\times \N_1^\Lambda$, one has
$$p_{ij}=\Prob\{Z(k+1)=s_j\,|\,Z(k)=s_i\}
     =\Prob\{N(k+1)=j\,|\,N(k)=i\}=q_{ij},$$
whereas, for all $j \in \N_1^{\Lambda-1}$, the transition probabilities satisfy
\begin{equation*}
   \begin{split}
p_{(j+1)(\Lambda+j)}&=\Prob\{Z(k+1)=s_{\Lambda
  +j}\,|\,Z(k)=s_{j+1}\}=\Prob\{Z(k+1)=(0,j)\,|\,Z(k)=(j+1,j+1)\}\\
&=\Prob\{N(k+1)=0\,|\,N(k)=j+1\}=q_{(j+1)0}.
\end{split}
\end{equation*}
Direct calculations also yield that for all $m\in\N_{2}^{\Lambda-1}$,
\begin{equation*}
   \begin{split}
p&_{(\Lambda+m)(\Lambda+m-1)}=
 \Prob\{Z(k+1)=s_{\Lambda+m-1}\,|\,Z(k)=s_{\Lambda+m}\}\\&=\Prob\{Z(k+1)=(0,m-1)\,|\,Z(k)=(0,m)\}=\Prob\{N(k+1)=0\,|\,N(k)=0\}=q_{00},
\end{split}
\end{equation*}
and, for all $(k,l)    \in\N_1^{\Lambda-1}\times \N_1^\Lambda$, it holds that
$p_{(\Lambda+k)l}=\Prob\{Z(k+1)=s_l\,|\,Z(k)=s_{\Lambda+k}\}=\Prob\{Z(k+1)=(l,l)\,|\,Z(k)=(0,k)\}
     =\Prob\{N(k+1)=l\,|\,N(k)=0\}=q_{0l}.$
Due to~(\ref{eq:19}), the other transitions will never  occur. 

\section{Proof of Lemma~\ref{lem:L2}}
\label{sec:proof-lemma-1}
The case $j=1$ is immediate, since
 $      \Prob\{\Delta_i=1\}=\Prob\{Z(k+1)=(0,0)\,|\,Z(k)=(0,0)\}
      =\Prob\{N(k+1)=0\,|\,N(k)=0\}=q_{00}$.
For $j\geq 2$, we proceed
  as follows: For all $s_i\in\mathbb{S}$, $i\not =0$,  denote by $\nu_i$  the
  first passage time of the state 
  $s_i$ to $s_0$. Thus, $\nu_i$ are random
  variables, with $\nu_i=j$ if the state $s_0$ is entered from $s_i$ for the
  first time in $j$ steps. Since only the states $s_{1}$ and
  $s_{\Lambda+1}$ can reach $s_0$ in one step, using~\eqref{eq:28},
  \begin{equation}
    \label{eq:25}
       \Prob\{\nu_i = 1\}= p_{i0}=
     \begin{cases}
      q_{10},&\text{if $i=1$}\\
      q_{00},&\text{if $i=\Lambda+1$}\\
      0,&\text{if $i\in\N_{2}^\Lambda \cup \N_{\Lambda +2}^{2\Lambda-1}$.}
    \end{cases}
\end{equation}  
For $j\geq 2$, paths from $s_i$ to $s_0$ go through intermediate states $s_\ell\not = s_0$,
providing the recursions
 $$ \Prob\{\nu_i = j\} =\sum_{\ell =1}^{2\Lambda-1} p_{i \ell } \Prob\{\nu_\ell =
 j-1\}, \quad \forall i\in\N_1^{2\Lambda-1},$$
which can be stated in matrix form via:
  \begin{equation*}
    \begin{bmatrix}
      \Prob\{\nu_1 = j\}\\
      \vdots\\
      \Prob\{\nu_{2\Lambda-1} = j\}
    \end{bmatrix}
  = \overline{\mathcal{P}}
    \begin{bmatrix}
      \Prob\{\nu_1 = j-1\}\\
      \vdots\\
      \Prob\{\nu_{2\Lambda-1} = j-1\}
    \end{bmatrix}
      =\overline{\mathcal{P}}^{j-1} \begin{bmatrix}
      \Prob\{\nu_1 = 1\}\\
      \vdots\\
      \Prob\{\nu_{2\Lambda-1} = 1\} 
    \end{bmatrix}= \overline{\mathcal{P}}^{j-1} \mu,
  \end{equation*}
which in view of~(\ref{eq:25}) and~(\ref{eq:28}), holds not only for $j\geq 2$ ,
but also 
for $j=1$.  The result now follows by using~(\ref{eq:25}) and the distribution
of $\{\Delta_i\}$. The latter  can be obtained from the distribution of $\nu_i$
  by considering the transitions 
   from $s_0$ to nodes other than itself (see~\eqref{eq:28}):
$$  \Prob\{\Delta_i=j\}=\sum_{\ell =1}^{2\Lambda-1} p_{0\ell} \Prob\{\nu_\ell =
    j-1\}=\sum_{\ell
      =1}^{\Lambda} q_{0\ell} \Prob\{\nu_\ell = j-1\}.$$

\section{Proof of Corollary~\ref{c14}}
\label{sec:proof-corollary}
For the situation of interest, the term
$\Upsilon_\varsigma$ introduced in Lemma 4 of\cite{quegup12a} can be written as:
\begin{equation}
  \begin{split}
    \label{eq:27}
    \Upsilon_\varsigma &=\bar q_\varsigma\big(I-\rho\bar{Q}\big)^{-1}\,  \bigg(    \rho I    + (\alpha-\rho)\big(I-\alpha\bar{Q}\big)^{-1}\sum_{l=1}^\Lambda    p_{l|\varsigma}(\rho\bar{Q})^l\bigg) \bar p\\
  \end{split}
\end{equation}
for all $\varsigma \in  \{2,3,\dots,\Lambda +1\}$
and where
\begin{equation*}
  \begin{split}
    \bar{q}_{\varsigma} &=
    \begin{bmatrix}
      q_{(\varsigma-1)0} & q_{(\varsigma-1)1} &\dots & q_{(\varsigma-1)\Lambda}
    \end{bmatrix},\quad
\bar Q =
      \begin{bmatrix}
        q_{00} & q_{01} &\dots & q_{0\Lambda}\\
        0 & 0 & \dots & 0\\
        \vdots & \vdots & \ddots & \vdots\\
         0 & 0 & \dots & 0
      \end{bmatrix}\!,\; \bar p =
    \begin{bmatrix}
      0 \\ 1 \\ \vdots \\ 1
    \end{bmatrix}
  \end{split}
\end{equation*}
Direct calculations lead to $\bar Q^l=(q_{00})^{l-1} \bar Q$, for all $l\in
\N$ and~\eqref{eq:27} is then condensed into
\begin{equation*}
  \begin{split}
    \Upsilon_\varsigma &=\bar q_\varsigma\big(I-\rho\bar{Q}\big)^{-1}   \Big( \rho I
    + \rho^{\varsigma-1}(q_{00})^{\varsigma-2} (\alpha-\rho)\big(I-\alpha\bar{Q}\big)^{-1} 
     \bar{Q}\Big) \bar p\\
     &=\frac{1}{1-q_{00}\rho}\Bigg(\frac{(\alpha-\rho)}{1-q_{00}\alpha}q_{00}^{\varsigma-2}\rho^{\varsigma-1}+\rho^{2}\Bigg)\sum_{l=1}^\Lambda
     q_{(\varsigma-1)0}q_{0l} +\rho \sum_{l=1}^\Lambda q_{(\varsigma-1)l},
  \end{split}
\end{equation*}
which proves the result.
  
\section{Proof of Theorem~\ref{thm:robust}}
\label{proof:thm:robust}
To establish this result, we first extend Lemma~\ref{lemma:anytime_inter} to the
perturbed plant case~\eqref{eq:51}. Clearly, for $\Delta_0\geq 1$
(and setting $k_0=0$, $x(0)=x$, and using notation $\E_x\{ \cdot\} = \E\{\cdot
\,|\, x(0)=x\}$), we have $u(0)=\mathbf{0}_p$. Thus,
\begin{equation*}
  \label{eq:37}
  V(x(1)) \leq \alpha V(x) +\eta |w(0)|,\quad \forall \Delta_0\geq 1
\end{equation*}
so that
\begin{equation*}
  \label{eq:30}
  \E_x \{V(x(1))\,|\,\Delta_0\geq 1\}\leq \alpha V(x) + W\Psi_1(0),  
\end{equation*}
where $W\eq\E\{|w(k)|\}$ and $\Psi_1(0)=\eta$. Now for $\Delta_0\geq 2$, thus $u(1)=\kappa((x(1))$,
using the above we obtain
\begin{equation}
  \label{eq:36}
  V(x(2)) \leq  \rho V(x(1)) +\beta |w(1)| \leq \rho \alpha V(x(0)) +\rho \eta
  |w(0)| +\beta |w(1)| , \quad \forall \Delta_0\geq 2
\end{equation}
yielding
\begin{equation*}
  \label{eq:38}
  \E_x\{V(x(2))\,|\, \Delta_0\geq 2 \}\leq \alpha \rho  V(x) +W\Psi_2(1),  
\end{equation*}
with $\Psi_2(1)=\rho \eta +\beta$.

\par For $\Delta_0\geq 3$ analyzing $V(x(3))$ becomes more involved since $u(2)$
could have been calculated using $x(1)$ or $x(2)$:
\begin{equation}
  \label{eq:40}
  u(2)\in\{\kappa(x(2)), \kappa (\bar f (x(1))) \}
\end{equation}
where
\begin{equation}
  \label{eq:49}
  \bar f (x )\eq f(x ,\kappa(x ), \mathbf{0}_n).
\end{equation}
For notational convenience, we let $\bar f^0(x)\eq x$ and define the result of
two repeated iterations of  ({\ref{eq:49}}) as $\bar f ^2(x )\eq f(\bar f (x )
,\kappa(\bar f (x ) ), \mathbf{0}_n)$, and $\bar f ^k(x )$ as the result of $k$
repeated iterations of   ({\ref{eq:49}}). Interestingly, due to
continuity, both cases in~\eqref{eq:40} are 
not that far away from $\kappa(x(2))$. In fact,
\begin{equation*}
  \label{eq:43}
  \begin{split}
   |u(2) - \kappa(x(2)) | &\leq |\kappa (\bar f (x(1))) - \kappa(x(2)) |
    \leq  \lambda_\kappa| \bar f (x(1)) - x(2)|\\
    &=  \lambda_\kappa |  f(x(1),\kappa(x(1)),\mathbf{0}_n) -
    f(x(1),\kappa(x(1)),w(1))| \leq   \lambda_\kappa \lambda_w|w(1)|.
  \end{split}
\end{equation*}
Thus,
\begin{equation*}
  \label{eq:44}
  \begin{split}
    &|f(x(2),u(2),w(2))-f(x(2),\kappa(x(2)),w(2))| \\
    &\leq |f(x(2),\kappa (\bar f
    (x(1))),w(2))-f(x(2),\kappa(x(2)),w(2))|\leq \lambda_u\lambda_\kappa
    \lambda_w|w(1)|
  \end{split}
  \end{equation*}
which using~\eqref{eq:36} gives
\begin{equation}
  \label{eq:46}
  \begin{split}
    V(x(3))&= V(f(x(2),\kappa(x(2)),w(2)))+V(x(3))- V(f(x(2),\kappa(x(2)),w(2)) )\\
    &\leq \rho V(x(2)) +\beta |w(2)| + \lambda_V \lambda_u\lambda_\kappa
    \lambda_w|w(1)|\\
    &\leq \alpha \rho^2  V(x(0)) + \rho^2\eta |w(0)|
 + \rho\beta  |w(1)|+\beta |w(2)| + \lambda_V \lambda_u\lambda_\kappa
    \lambda_w|w(1)|, \quad \forall \Delta_0\geq 3
  \end{split}
\end{equation}
and
\begin{equation*}
   \E_x\{ V(x(3))\,|\, \Delta_0\geq 3\}\leq  \alpha  \rho^2  V(x ) + W \Psi_3(2)
\end{equation*} 
with
\begin{equation*}
  \Psi_3(j) = \sum_{\ell=0}^j \psi_{3,\ell}\rho^\ell,\quad 
   \psi_{3,0}  =\beta
   + \lambda_V \lambda_u\lambda_\kappa \lambda_w,\; \psi_{3,1} =\beta, \; \psi_{3,2} =\eta,
\end{equation*}

\par 
To extend the above analyis to   $V(x(4))$ for $\Delta_0\geq 4$, simply note that (for $\Lambda \geq 3$), 
\begin{equation*}
  u(3)\in\big\{\kappa(x(3)),\kappa(\bar f(x(2))), \kappa \big(\bar
  f (\bar f (x(1)))\big) \big\},
\end{equation*}
thus,
\begin{equation*}
  |u(3) - \kappa(x(3)) |\leq \max\big\{ |\kappa (\bar f(x(2))) - \kappa(x(3)) | ,|\kappa (\bar f^2 (x(1))) - \kappa(x(3)) |\big\}
\end{equation*}

By examining both cases separately, an upper bound can be obtained. For example,
for $u(3)=\kappa (\bar f^2 (x(1)))$, we have
\begin{equation*}
  \begin{split}
    |\kappa (\bar f^2 (x(1))) - \kappa(x(3)) |
    &\leq  \lambda_\kappa| \bar f^2 (x(1)) - x(3)|\\
    &\leq  \lambda_\kappa |  f(\bar{f} (x(1)),\kappa(\bar{f} (x(1)),\mathbf{0}_n) -
    f(x(2),\kappa(x(2)),w(2))|\\
    &\leq   \lambda_\kappa \bigg{[}(\lambda_x+\lambda_u\lambda_{\kappa})|\bar{f} (x(1)-x(2)|+ \lambda_w|w(2)|\bigg{]}\\
    &\leq  \lambda_\kappa(\lambda_x+\lambda_u\lambda_{\kappa})\lambda_w|w(1)|+\lambda_{\kappa}\lambda_w|w(2)|.
  \end{split}
\end{equation*}
leading to
\begin{equation*}
  \begin{split}
    &|f(x(3),u(3),w(3))-f(x(3),\kappa(x(3)),w(3))|\\
     &\leq |f(x(3),\kappa (\bar f^2
    (x(1))),w(3))-f(x(3),\kappa(x(3)),w(3))|\\
    &\leq \lambda_u\bigg{[}\lambda_\kappa(\lambda_x+\lambda_u\lambda_{\kappa})\lambda_w|w(1)|+\lambda_{\kappa}\lambda_w|w(2)|\bigg{]}
  \end{split}
  \end{equation*}
which gives (using~\eqref{eq:46})
\begin{equation*}
  \begin{split}
    V(x(4))&= V(f(x(3),\kappa(x(3)),w(3)))+V(x(4))- V(f(x(3),\kappa(x(3)),w(3)) )\\
    &\leq \rho V(x(3)) +\beta |w(3)| + \lambda_V \lambda_u\bigg{[}\lambda_\kappa(\lambda_x+\lambda_u\lambda_{\kappa})\lambda_w|w(1)|+\lambda_{\kappa}\lambda_w|w(2)|\bigg{]}\\
    &\leq \alpha \rho^3  V(x(0)) + \rho^3\eta |w(0)|
 + \bigg(\rho^2\beta+\rho\lambda_V\lambda_u\lambda_\kappa\lambda_w+\lambda_V \lambda_u\lambda_\kappa(\lambda_x+\lambda_u\lambda_{\kappa})\lambda_w\bigg)|w(1)|\\
 &+(\rho\beta+\lambda_V\lambda_u\lambda_\kappa\lambda_w)|w(2)| +\beta|w(3)|, \quad \forall \Delta_0\geq 4.
  \end{split}
\end{equation*}

For the case $u(3)=\kappa (\bar f(x(2)))$, a similar expression can be obtained,
leading to a common upper-bound of the form:
\begin{equation*}
   \E_x\{ V(x(4))\,|\, \Delta_0\geq 4\} \leq \alpha \rho^3  V(x(0)) + W \Psi_4(3), 
\end{equation*}
where  $\Psi_4(j)=\sum_{\ell =0}^{j}\psi_{4,\ell}\rho^\ell$, $j\leq 3$. 

To continue the analysis presented above,
for $\Delta_0\geq j,~j\geq 5$ (for $\Lambda \geq j-1$),  we note that
\begin{equation*}
  u(j-1)\in\big\{\kappa(x(j-1)),\kappa(\bar f(x(j-2))),..., \kappa \big(\bar
  f^{l-1} (x(j-l))\big),~l=1,...,j-1 \big\},
\end{equation*}
Following similar ideas,  one obtains
\begin{equation}
\label{eq:1147}
   \E_x\{ V(x(j))\,|\, \Delta_0\geq j\} \leq \alpha \rho^{j-1}  V(x(0)) + W \Psi_j(j-1), 
\end{equation}
where  $\Psi_j(j-1)=\sum_{\ell =0}^{j-1}\psi_{j,\ell}\rho^\ell$.
Notice that, since the buffer length $\Lambda$ is bounded,  the terms
$\psi_{j,\ell},~\forall j\in \N,~
\ell\in\N_{0}^{j-1}$ are bounded.
 
The above analysis allows one to  generalize Lemma \ref{lemma:anytime_inter} to
the case with i.i.d. disturbances.  The law of total expectation, the fact that
$\{\Delta_i\}$ is i.i.d., and expression~(\ref{eq:1147}) give
\begin{equation}
    \label{eq:66b}
    \E\big\{V(x({k_{1}}))\,\big|\,x(k_{0})=\chi\big\}\\
 \leq \Omega V(\chi)+\Sigma,\quad 
 \forall \chi\in \R^n,
\end{equation}
where\footnote{Note that, since $W$ is assumed bounded, $\Psi_j(j-1)$,
  $j\in\N$  are
uniformly bounded and, thus, $\Sigma$ is bounded.} $$\Sigma\eq \sum_{j\in\N}\Prob\{\Delta_i=j\} W \Psi_j(j-1).$$  
 
From~(\ref{eq:3}) and~(\ref{eq:66b}) and since  Lemma~\ref{lem:Markov} holds
also in the perturbed case, it
follows that if $\Omega  <1$, then for all
    $(i,\chi_0)\in\N_0\times\R^n$,
  $$  \E\{V(x({k_{i}}))\,|\,x(k_{0})=\chi_0\}
    \leq\Omega^{i}V(\chi_0)+\big(\Omega^{i-1}+\Omega^{i-2}+...+\Omega+1\big)\Sigma.$$
    For the time steps $k\in \N \,\backslash \, \mathcal{K}$, i.e., where
   calculated control values are applied,~(\ref{eq:1147}) and the law of total
   expectation yield
  \begin{equation*}
 \E\Bigg\{\sum_{k=k_{i}}^{k_{i+1}-1}V(x({k}))\,\bigg|\,x(k_{i})=\chi_i\Bigg\}  \leq \bigg(1
 +\frac{\alpha}{1-\rho}\bigg)V(\chi_i)+W',
\end{equation*}
 where
 \begin{equation*}
   W'\eq
   \sum_{j\geq 2}\Prob\{\Delta_i=j\}W \Psi_j(j-2)  <\infty,
 \end{equation*}
with $\Psi_j(j-2)=\sum_{\ell =0}^{j-2}\psi_{j,\ell}\rho^\ell$.
Taking conditional expectation $\E\{\, \cdot\,|\, x(k_0)=\chi_0\}$ on both
sides, defining 
$\beta \eq  (1+\alpha-\rho)/(1-\rho)$  and using the Markovian property of $\{x\}_{\K}$ yields
  \begin{equation*}
\begin{split}
 & \E\Bigg\{ \E\Bigg\{\sum_{k=k_{i}}^{k_{i+1}-1}V(x({k}))\,\bigg|\, x(k_{i})=\chi_i\Bigg\}
    \,\Bigg|\, x(k_0)=\chi_0\Bigg\} 
    = \E\Bigg\{\sum_{k=k_{i}}^{k_{i+1}-1}V(x({k}))\,\bigg|\,
    x(k_{0})=\chi_0\Bigg\}\\
    &\leq   \beta
    \E\big\{V(x({k_{i}}))\,\big|\,x(k_0)=\chi_0\big\}+ W' \\
    &\leq \beta
    \Omega^i V(\chi_0)+\beta\bigg(\Omega^{i-1}+\Omega^{i-2}+...+\Omega+1\bigg)\Sigma
    +W' \end{split}
\end{equation*}
Thus,
\begin{equation*}
  \E\Big\{V(x({k}))\,\big|\,
    x(k_{0})=\chi_0\Big\} \leq \beta
    \Omega
    V(\chi_0)+\beta\bigg(\sum_{\ell=0}^{i-1}\Omega^{\ell}\bigg)\Sigma+W'  ,\quad \forall
    k\in\{k_i,k_i+1\dots,k_{i+1}-1\}. 
\end{equation*}
 
Now recall that $k_0=0$ and that
\begin{equation*}
  \bigcup_{i\in\N_0} \{k_i,k_i+1\dots,k_{i+1}-1\} = \N_0.
\end{equation*}
The result now follows by using~(\ref{eq:3}), Assumption~\ref{ass:bound_prob}
and taking 
expectation with respect to the distribution of $x(0)$.

\end{document}